\documentclass[journal]{IEEEtran}

\ifCLASSINFOpdf
    \usepackage[pdftex]{graphicx}
\else
    \usepackage[dvips]{graphicx}
\fi
\ifCLASSOPTIONcompsoc
\usepackage[caption=false,font=normalsize,labelfont=sf,textfont=sf]{subfig}
\else
  \usepackage[caption=false,font=footnotesize]{subfig}
\fi

\usepackage{amsmath}
\usepackage{amsthm}
\usepackage{amsfonts}
\usepackage{amssymb}
\usepackage{bm}
\usepackage{lipsum}  
\usepackage{algorithmic}
\usepackage{array}
\usepackage{url}

\usepackage{todonotes}

\usepackage{multirow}
\usepackage{hhline}
\usepackage{diagbox}
\usepackage{arydshln}
\usepackage{algorithm}
\usepackage{cite}

\DeclareMathAlphabet\mathbfcal{OMS}{cmsy}{b}{n}
\newcommand{\vect}[1]{\boldsymbol{\bm{#1}}}
\newcommand{\mat}[1]{\boldsymbol{\mathbf{#1}}}

\newcommand{\diff}{\, \mathrm d}
\newcommand{\divs}{\vect \nabla_\Gamma \cdot}

\newtheorem{proposition}{Proposition}
\newtheorem{corollary}{Corollary}

\begin{document}
%
\title{High-order quasi-Helmholtz Projectors: 
\\ Definition, Analyses, Algorithms}


%
%
%

\author{Johann~Bourhis,        Adrien~Merlini,~\IEEEmembership{Member,~IEEE,}
        and~Francesco~P.~Andriulli,~\IEEEmembership{Fellow,~IEEE}
\thanks{J.~Bourhis and F. P. Andriulli are with the Department of Electronics
and Telecommunications, Politecnico di Torino, 10129 Torino, Italy; e-mail:
francesco.andriulli@polito.it.}
\thanks{A. Merlini is with the Microwave department, IMT Atlantique, 29238
Brest cedex 03, France; e-mail: adrien.merlini@imt-atlantique.fr.}
\thanks{Manuscript received August 25, 2023.}}

%
%

\markboth{}%
{}
%



\maketitle

\begin{abstract}
The accuracy of the electric field integral equation (EFIE) can be substantially improved using high-order discretizations. 
However, this equation suffers from ill-conditioning and deleterious numerical effects in the low-frequency regime, often jeopardizing its solution. 
This can be fixed using quasi-Helmholtz decompositions, in which the source and testing elements are separated into their solenoidal and non-solenoidal contributions, then rescaled in order to avoid both the low-frequency conditioning breakdown and the loss of numerical accuracy.
However, standard quasi-Helmholtz decompositions require handling discretized differential operators that often worsen the mesh-refinement ill-conditioning and require the finding of the topological cycles of the geometry, which can be expensive when modeling complex scatterers, especially in high-order.
This paper solves these drawbacks by presenting the first extension of the quasi-Helmholtz projectors to high-order discretizations and their application to  the stabilization of the EFIE when discretized with high-order basis functions.
Our strategy will not require the identification of the cycles and will provide constant condition numbers for decreasing frequencies. Theoretical considerations will be accompanied by numerical results showing the effectiveness of our method in complex scenarios.
\end{abstract}


%
\IEEEpeerreviewmaketitle

\section{Introduction}
%
%
%
%
\IEEEPARstart{M}{odelling} and simulation of electromagnetic scattering from perfectly electrically conducting (PEC) objects can be effectively performed using surface integral equations~\cite{Chew2001FastAE, gibson_2021}. Among these formulations, the electric field integral equation (EFIE) is one of the most widely used.
This equation is usually solved via the boundary element method (BEM)~\cite{sauter2010boundary} by approximating the current as a combination of basis functions with a finite support defined on a surface mesh of the object.
The accuracy of the method consequently depends on the ability of the mesh and basis functions to respectively describe the surface and the current functional space.
Very often, zeroth-order basis functions such as the Rao-Wilton-Glisson (RWG) functions are employed with flat triangular cells~\cite{rao-wilton-glisson-1982}.
Thus, the mesh density and the functional space discretization are usually increased simultaneously by refining the mesh, which leads to a higher number of cells and basis functions.

Alternatively, high-order mesh and functional space discretizations~\cite{graglia_higher_1997,graglia_higher-order_2015} can be employed to improve the accuracy without necessarily increasing the mesh density. 
High-order basis functions also provide a faster convergence to the physical solution when refining the mesh~\cite{djordjevic_double_2004, Peterson_error}.
Nevertheless, despite the use of a more accurate framework, the EFIE suffers from ill-conditioning and loss of significant digits at low-frequency~\cite{Adrian2021}.

The stabilization of the EFIE using quasi-Helmholtz decomposition~\cite{Vecchi1999LoopstarDO,Chew2000,Lee2003,Eibert} is well-known, even in the high-order case~\cite{wildman_accurate_2004,valdes_high-order_2011}.
It consists in a change of basis that allows to reorganize the system into solenoidal and non-solenoidal contributions and to rescale them appropriately to cure the problematic behavior of the EFIE at low-frequency.
However, the construction of the solenoidal basis functions can be burdensome because it requires the identification of the global cycles;
a task that might be challenging when modeling complex geometries~\cite{Adrian2021}.

More recently, the method of the quasi-Helmholtz projectors has been developed for the RWG case~\cite{Adrian2021}. 
The method generates orthogonal projectors over the solenoidal and non-solenoidal subspaces from the computation of the Star matrix, without having to explicitly identify the cycles. 
They are subsequently used to separate and to rescale the different contributions of the EFIE, without further degrading the condition number when increasing the mesh-density~\cite{Adrian2021}.

This work proposes for the first time the high-order counterpart of the quasi-Helmholtz projectors for the stabilization of the EFIE when using high-order basis functions. The contribution will first present a general definition of the Star basis functions in high-order that encompasses several choices in terms of basis elements and testing in the charge space. The contribution will then show that the consequential definition of the quasi-Helmholtz projectors is not dependent on any of these specific choices. Finally the new projectors will be used to regularize the electric field integral equation discretized with high-order elements.

This paper is organized as follows. 
In Section~\ref{sec:background}, we set the background and the notations.
In section~\ref{contribution}, we extend the definition of the quasi-Helmholtz projectors to the high-order framework. We show their completeness to represent the Graglia-Wilton-Peterson (GWP) basis~\cite{graglia_higher_1997} and their uniqueness with respect to the choice of the Star basis, as well as their application to solve the low-frequency breakdown.
In section~\ref{implementation}, we present the implementation details required to achieve effective algorithms. 
Finally, section~\ref{numerical-results} provides numerical results which validate the preconditioning technique on relevant scenarios.
The results of this work were presented in a conference contribution~\cite{URSIGASS2023} that generalized the preliminary investigations in~\cite{merlini2019unified}.

\section{Background and Notations}\label{sec:background}

This section will introduce the necessary background material and notation on integral equations and related high-order discretizations. The treatment will be brief and with the primary goal of setting the notations. The interested reader is referred to the more extensive treatments in~\cite{graglia_higher_1997, Peterson2005, Nedelec1980, Brezzi2013, djordjevic_double_2004} and references therein.

Consider the scattering from a PEC object with boundary $\Gamma$ residing in a homogeneous medium with wavenumber $k$ and characteristic impedance $\eta$. The current $\vect J$ induced on $\Gamma$ by an incident electric field $\vect E^{\mathrm{i}}$ can be obtained by solving the EFIE~\cite{gibson_2021}
\begin{equation}\label{EFIE}
    \mathcal{T} \vect J =  -\widehat{\vect n} \times \dfrac{1}{\eta} \vect E^{\mathrm{i}},
\end{equation}
with
\begin{equation}
    \mathcal T = j k \, \mathcal T_s - \frac{1}{j k} \, \mathcal T_h,\\
\end{equation}
where the vector and scalar potentials are defined as
\begin{gather}
    \left( \mathcal T_s \vect J \right)(\vect r) = \widehat{\vect n} \times \int_{\Gamma} G(\vect{r},\vect{r'}) \vect J(\vect{r'}) \diff S(\vect r'), \\
    \left(    \mathcal T_h \vect J \right)(\vect r) =\widehat{\vect n} \times \vect \nabla_\Gamma \int_{\Gamma} G(\vect{r},\vect{r'}) \vect \nabla_\Gamma \cdot \vect J(\vect{r'}) \diff S(\vect r'),
\end{gather}
and where $\widehat{\vect n}$ is the outgoing normal vector from $\Gamma$, $\vect \nabla_\Gamma$ is the surface nabla operator, and the Green's kernel is 
\begin{equation}
    G(\vect r, \vect r') = \dfrac{e^{-j k |\vect r-\vect r'|}}{4 \pi |\vect r-\vect r'|}.
\end{equation}

The solution  $\vect J$  of \eqref{EFIE} lives in
the functional space $H^{-\frac{1}{2}}_{\mathrm{div}}\left(\Gamma\right)$, for which a  well-suited arbitrary order discretization  approach is given by the Nédélec's mixed-order div-conforming spaces~\cite{Nedelec1980, Brezzi2013}.
These spaces can be generated by different sets of basis functions\cite{graglia_higher_1997, notaros_higher_2008, Peterson2005}. 
In the following, for fixing ideas, we will focus on the Graglia-Wilton-Peterson (GWP) basis function set, but our considerations and findings will apply also to several other bases.
The $p$-th order GWPs are defined as the product of the RWGs (zeroth-order GWPs) with order $p$ shifted Silvester-Lagrange functions~\cite{graglia_higher_1997}.
In what follows, we denote by $\{\vect \psi_n^{(p)}\}_{n=1}^{N_p}$ the set of GWP basis functions, where the total number of functions is 
\begin{equation}\label{numGWPs}
    N_p = (p+1) E_\mathrm{int} + p (p+1) C,
\end{equation}
with $E_\mathrm{int}$ and $C$ the number of internal edges and cells of the mesh. A detailed definition of these basis functions is omitted here for space limitation, but we refer the reader to \cite{graglia_higher_1997} and \cite{Peterson2005}.\color{black}

\begin{color}{black}
\color{black}The space spanned by the GWPs' divergences is included into the space of cell-wise polynomials complete up to order $p$~\cite{graglia_higher-order_2015}, \color{black}that is of dimension 
\begin{equation}
    M_p = \dfrac{(p+1)(p+2)}{2} C,
\end{equation}
and that we denote by $\mathbb P(\Gamma,p)$ in what follows.
We consider on each triangle a set of $\frac{(p+1)(p+2)}{2}$ interpolatory points, which give a total set of points $\{\vect r_j\}_{m=1}^{M_p}$ for the whole mesh \color{black}(with overlapping for the points defined on the triangles' boundaries).
On these points, we define a Lagrange interpolatory basis 
\color{black}for $\mathbb P(\Gamma,p)$ that we denote by $\{\sigma^{(p)}_m\}_{m=1}^{M_p}$. 
More explicitly we have
\begin{equation}\label{eq:sigmadef}
    \sigma_m(\vect r_j) = \delta_{mj},
\end{equation}
\color{black} with $\delta_{mj}$ the Kronecker symbol and $\sigma_m$ is a piecewise continuous function which is polynomial of degree $p$ on one cell and zero on all the others. 

\color{black} Because of charge neutrality, the total number of degrees of freedom (DoFs) for the divergence space (charge DoFs)
is $M_p - N_\mathrm{bodies}$~\cite{wildman_accurate_2004}, where $N_\mathrm{bodies}$ is the number of separate connected bodies of the mesh. 
The GWP space can be decomposed as a 
combination of non-solenoidal functions, corresponding to the charge DoFs, and a combination of solenoidal (divergence free) functions that complement the charge DoFs.
\end{color}

The GWP functions $\{\vect \psi^{(p)}_n\}$ can be used within a BEM strategy, by approximating 
the current in \eqref{EFIE} as 
\begin{equation}\label{GWPsCombination}
    \vect J(\vect r) \approx \sum_{n=1}^{N_p} j_n \vect \psi^{(p)}_n(\vect r)
\end{equation} 
and, after testing \eqref{EFIE} with the rotated GWP elements~$\{ \widehat{ \vect n} \times \vect \psi^{(p)}_n\}$, the EFIE yields a matrix system~
\begin{equation}
    \mat T \vect j = \vect e,
\end{equation}
with
\begin{equation}
    \mat T = j k \, \mat T_s + \frac{1}{j k} \, \mat T_h\,, \label{system-EFIE}
\end{equation}
where
\begin{gather}
    \big[\mat T_s\big]_{mn} = \int_{\Gamma \times \Gamma} \hspace{-0.45cm} G(\vect{r},\vect{r'}) \vect\psi^{(p)}_n(\vect{r'}) \cdot \vect\psi^{(p)}_m(\vect{r}) \diff {S(\vect r')} \mathrm d {S(\vect r)},\label{Ts-matrix} \\
    \big[\mat T_h\big]_{mn} = \int_{\Gamma \times \Gamma} \hspace{-0.45cm} G(\vect{r},\vect{r'}) \vect{\nabla}_\Gamma \cdot \vect\psi^{(p)}_n(\vect{r'}) \vect{\nabla}_\Gamma \cdot \vect\psi^{(p)}_m(\vect{r}) \diff {S(\vect r')} \mathrm d {S(\vect r)},\label{Th-matrix} \\
    \big[\vect e \big]_m = -\frac{1}{\eta} \int_\Gamma \vect{E}^{\mathrm{i}} (\vect{r}) \cdot \vect\psi^{(p)}_m (\vect{r}) \diff {S(\vect r)}.
\end{gather}




At low-frequency, the EFIE faces  several numerical challenges. 
First, the EFIE linear system becomes increasingly ill-conditioned for decreasing frequencies due to the frequency ill-scaling of vector and scalar potentials~\cite{Adrian2021}. 
In particular
\begin{equation}
    \lim_{k \rightarrow 0} \mathrm{cond}(\mat T) = \mathcal O(k^{-2}).
\end{equation}
This ill-conditioning impacts the accuracy of the solution and increases the number of iterations required by iterative solvers.
A second numerical challenge is related to the loss of significant digits in the context of finite precision computations when evaluating the right-hand-side $\vect e$, the solution $\vect j$ and the radiated field~\cite{Adrian2021}.

Both effects are related to the quasi-Helmholtz decomposition of the current, because solenoidal and non-solenoidal contributions in the EFIE system have different frequency behaviors.
By decomposing the source and testing elements into their solenoidal and non-solenoidal components, it is possible to properly rescale these contributions to cure both of these issues. In the zeroth-order case, this strategy can be applied in a particularly effective way by leveraging quasi-Helmholtz projectors \cite{Adrian2021}. Differently from other quasi-Helmholtz decompositions such as Loop-Star/Tree and related approaches, this mathematical tool allows for the rescaling of scalar and vector potential without perturbing the other conditioning properties of the equation. The generalization of quasi-Helmholtz projector strategies to the high-order case, however, is far from trivial because of the need for a proper extension of the graph Laplacian matrices~\cite{Adrian2021} to the high-order case. Such a generalization will be the subject of Section~\ref{contribution}.



\section{High-Order Quasi-Helmholtz Projectors}\label{contribution}

In this section we will define high-order quasi-Helmholtz projectors to  address the above described low-frequency limitations of the EFIE in the high-order case. From the zeroth-order case \cite{Adrian2021} we learn that primal projectors can be obtained from a properly chosen \emph{Star matrix}.
The problem is that a unique definition of a Star matrix in high-order can be challenging~\cite{wildman_accurate_2004}. In this contribution, we will propose one approach to define a Star matrix that, indeed, does not lead to a unique definition. We will equally show in this paper, however, that the corresponding Star (non-solenoidal) projector will be invariant, regardless of the non-uniqueness of the Star definition we will adopt.

\subsection{Construction of the high-order quasi-Helmholtz projectors}

Given the coefficients $\vect j$ of a function expressed as a linear combination of GWP functions, one could think of building a high-order Loop-Star decomposition in the form
\begin{equation} \label{eq:dishelmholtz}
    \vect j= \mat \Sigma_p \vect s+\mat \Lambda_p \vect l + \mat H_p \vect h.
\end{equation}
Similarly to the zeroth-order case $\mat \Lambda_p$ and $\mat H_p$ express the local-Loops-to-GWP and global-cycles-to-GWP change of bases, respectively \cite{Adrian2021}. 
In the following we will not need an explicit definition of these two matrices and we will not discuss them further. We will only be using the fact that both $\mat \Lambda_p$ and $\mat H_p$ contain coefficient representations of solenoidal functions. 
As pertains to $\mat \Sigma_p$, we will refer to this matrix as the high-order \emph{Star} matrix, with an abuse of terminology steming from the zeroth-order case. Consider now the injective linear function $\mathcal L$
\begin{equation}\label{eq:defofL}
    \mathcal L\colon \begin{array}{>{\displaystyle}r @{} >{{}}c<{{}} @{} >{\displaystyle}l} 
          &\mathbb P(\Gamma,p) \longrightarrow& \mathbb R^m \\ 
          &\quad q(\vect r) \;\;  \longmapsto& \mathcal L q  
         \end{array}
         \end{equation}
with $[\mathcal L q]_m=\mathcal L_m q$. It should be noted that injectivity and linearity imply
\begin{equation}\label{eq:extraproperty}
\mathcal L q = 0 \iff q(\vect r) = 0 \quad \forall \vect r \in \Gamma. 
\end{equation}
Now we can propose the following general definition for $\mat \Sigma_p$
\begin{equation}\label{sigma-definition}
    \left[\mat \Sigma_p\right]_{i,j} =[\mathcal L \divs \vect \psi^{(p)}_i]_j
\end{equation}
By the construction of $\mat \Sigma_p$ and because $\mat \Lambda_p$ and $\mat H_p$ describe solenoidal functions we get
\begin{equation}
\mat \Sigma_p^T\mat \Lambda_p=\mat 0 
\end{equation}
and 
\begin{equation}
\mat \Sigma_p^T\mat H_p=\mat 0.
\end{equation}
The above conditions result in that all coefficients of solenoidal functions are in the null-space of $\mat \Sigma_p^T$. Also the converse statement, all elements of the null space of $\mat \Sigma_p^T$ are coefficients of solenoidal functions, is true following from \eqref{eq:extraproperty} and \eqref{sigma-definition}.
Since the solenoidal subspace of the GWPs has a dimension $N_p - M_p + N_{\mathrm{bodies}}$ \cite{wildman_accurate_2004}, the null-space of $\mat \Sigma_p^T$ has dimension $N_p - M_p + N_{\mathrm{bodies}}$. As a consequence, the dimension of the range of $\mat \Sigma_p^T$, which is also the dimension of the range of $\mat \Sigma_p$, equals to $N_p-(N_p - M_p + N_{\mathrm{bodies}})=M_p - N_{\mathrm{bodies}}$.
This fact proves that $\left[ \mat \Sigma_p, \mat \Lambda_p,  \mat H_p\right]$, a rectangular matrix, has a rank equal to the number of basis functions $N_p$ in \eqref{numGWPs}. This shows that our definition of the high-order Stars $\mat \Sigma_p^T$ produces a valid complement of the solenoidal subspaces as the columns of $\left[ \mat \Sigma_p, \mat \Lambda_p,  \mat H_p\right]$  generate the entire GWP space.

We are now ready to define the quasi-Helmoltz projectors. In particular, mimicking the zeroth-order case, define
\begin{equation}\label{p-nsol-proj}
    \mat P_p^{\Sigma} = \mat \Sigma_p \left(\mat \Sigma_p^T \mat \Sigma_p\right)^+ \mat \Sigma_p^T
\end{equation}
and
\begin{equation}\label{p-sol-proj}
    \mat P_p^{\Lambda H} = \mat I - \mat P_p^\Sigma.
\end{equation}
It should be noted that, because of the orthogonality of the projectors, i.e.
\begin{equation}\label{orthogonality}
    \mat P_p^{\Sigma}\mat P_p^{\Lambda H}=\mat P_p^{\Sigma}(\mat I - \mat P_p^\Sigma)=\mat P_p^{\Lambda H}\mat P_p^{\Sigma}=\mat P_p^{\Sigma}-\mat P_p^{\Sigma}=\mat 0,
\end{equation}
and because of completeness of the decomposition, $\mat P_p^{\Lambda H}$ is the projector in the solenoidal (local and global Loops) subspace.

From the generality of our definition of $\mat\Sigma_p$ in \eqref{sigma-definition} one could think of obtaining a different set of projectors for each specific choice
$\mat\Sigma_p$. We will now prove that, instead, the high-order quasi-Helmholtz projectors, with our definition, are invariant under any particular choice of the high-order Star matrix. In particular
\begin{proposition}
Consider  two different operators $\mathcal L$ and $\widetilde {\mathcal L}$ both satisfying definition and properties in \eqref{eq:defofL} with associated matrices $\mat \Sigma_p$ and $\widetilde{\mat \Sigma}_p$, respectively defined via \eqref{sigma-definition}, then there exists an invertible matrix $\mat M$ so that $\widetilde{\mat \Sigma}_p =\mat \Sigma_p \mat M$.
\end{proposition}
\begin{proof}
Consider an arbitrary basis $\left\{ b_k \right\}_{k=1}^{M_p}$ of  $\mathbb P(\Gamma,p)$. We can then express the divergence of each $\vect \psi_i^{(p)}(\vect r)$ as
\begin{equation}\label{eq:basisexpansion}
	\vect \nabla \cdot \vect \psi_i^{(p)}(\vect r) = \sum_{k=1}^{M_p} a_{ik} b_k(\vect r)
\end{equation}
with a proper set of coefficients $\left\{ a_{i,k} \right\}_{k=1}^{M_p}$ for each $\vect \psi_i^{(p)}(\vect r)$.
Since both $\mat \Sigma_p$ and $\widetilde{\mat \Sigma}_p$ are defined on the same GWP set 
$\vect \psi_i^{(p)}(\vect r)$ and assuming, without loss of generality, the same ordering of the basis functions in both cases, from \eqref{eq:basisexpansion} we have 
\begin{equation}
    \mat \Sigma_p = \mat A \mat L \text{ and } \widetilde{\mat \Sigma}_p = \mat A \widetilde{\mat L}
\end{equation} 
 with $\left[\mat A \right]_{ik} = a_{ik}$,	 $\left[\mat L \right]_{kj} = \mathcal L_j b_k$, and $\left[\widetilde{\mat L} \right]_{kj} = \widetilde{\mathcal L}_j b_k$.

Using the definition, linearity, and injectivity of $\mathcal L$, together with the completeness of the basis $\left\{ b_k \right\}_{k=1}^{M_p}$, we obtain
\begin{align}
	\mat L^T \vect x &= \vect 0 \nonumber
 \Rightarrow \sum_{k=1}^{M_p} \big( \mathcal L_m b_k \big) x_k = 0 ~~ \forall m\\
	&\Rightarrow \mathcal L_m \left( \sum_{k=1}^{M_p} x_k b_k \right) = 0 ~~ \forall m\\
	&\Rightarrow \sum_{k=1}^{M_p} x_k b_k = 0
	\Rightarrow \vect x = \vect 0,\nonumber
\end{align}
which, with linearity, establishes the invertibility of $\mat L$. The same strategy shows the invertibility of $\widetilde{\mat L}$.
We now have
\begin{align}
\mat A =  \mat A \Rightarrow \mat A \widetilde{\mat L}\widetilde{\mat L}^{-1} 
= \mat A \mat L \mat L^{-1}
\Rightarrow  \widetilde{\mat \Sigma}_p \widetilde{\mat L}^{-1}= 
 \mat \Sigma_p {\mat L}^{-1}
\end{align}
from which we obtain the looked for relationship
\begin{equation}
	\widetilde{\mat \Sigma}_p = \mat \Sigma_p  \mat M
\end{equation}
with $\mat M = \mat L^{-1} \widetilde{\mat L}$, an invertible matrix. 
\end{proof}
The  well-posedness of the definition of the projectors in \eqref{p-nsol-proj} and \eqref{p-sol-proj}) now follows. In particular
\begin{corollary}
The quasi-Helmholtz projectors are invariant on any specific choice of an high-order Star matrix satisfying \eqref{sigma-definition}.
\end{corollary}
\begin{proof}
Given two matrices as above, we have from the previous proposition that
$\widetilde{\mat \Sigma}_p = \mat \Sigma_p  \mat M$ with $\mat M$ unique and invertible. 
First note that, letting
\begin{equation}
\mat \Sigma_p =\mat U \mat S \mat V^T
\end{equation}
be the reduced Singular Value Decomposition (SVD) of $\mat \Sigma$ (i.e. the SVD where $\mat U$ and $\mat V$ are vertically  rectangular and column full rank matrices), then
\begin{equation}
\mat M^T \mat \Sigma_p^T\mat\Sigma_p\mat M=\left(\mat M^T \mat V\mat S^T \right) \left( \mat S\mat V^T\mat M\right)
\end{equation}
where the left matrix in parenthesis is a full column rank matrix and the right matrix in parenthesis is a full row rank matrix. By using the properties of pseudoinverses \cite{ben2003generalized} we get
\begin{equation}
\left(\mat M^T \mat \Sigma_p^T\mat\Sigma_p\mat M\right)^+= \left( \mat S\mat V^T\mat M\right)^+\left(\mat M^T \mat V\mat S^T \right)^+
\end{equation}
from which it  follows that
\begin{align}
    \mat P_p^{\widetilde{\Sigma}} &= \widetilde{\mat \Sigma}_p \left(\widetilde{\mat \Sigma}_p^T \widetilde{\mat \Sigma}_p\right)^+\widetilde{ \mat \Sigma}_p^T\nonumber
    \\&=     \mat \Sigma_p  \mat M \left(  \mat M^T \mat \Sigma_p^T\mat \Sigma_p  \mat M\right)^+\mat M^T\mat \Sigma_p^T\\
    &=\mat U \mat S \mat V^T \mat M \left( \mat S\mat V^T\mat M\right)^+\left(\mat M^T \mat V\mat S^T \right)^+\mat M^T \mat V \mat S^T \mat U^T\nonumber\\
     &=\mat U\mat U^T=\mat \Sigma_p   \left(  \mat \Sigma_p^T\mat \Sigma_p \right)^+\mat \Sigma_p^T=
     \mat P_p^{\Sigma}\,.\nonumber
       \end{align}
The uniqueness of $\mat P_p^{{\Sigma}}$ combined with \eqref{p-sol-proj} directly implies that of $\mat P_p^{\Lambda H}$.
\end{proof}

\subsection{Projector Based Solution to the Low Frequency Breakdown}
To fix ideas and obtain an algorithm for the projectors, we propose to choose the following explicit definition of the Star matrix $\mat \Sigma_p$~\cite{merlini2019unified}
\begin{equation}\label{ho-charge-matrix}
    \left[ \mat \Sigma_p^T \right]_{mn} = \int_\Gamma \sigma_m^{(p)}(\vect r) \vect \nabla_\Gamma\cdot \vect \psi_n^{(p)}(\vect r) \diff S(\vect r)
\end{equation}
where the functions $\sigma_m^{(p)}$ are defined in \eqref{eq:sigmadef}. 
The reader should note that this definition results in a standard Star matrix in the zeroth-order ($p=0$) case and generalizes the concept for higher orders. Other choices however could have been made without modifying the final results, as proved above. 
Now, from \eqref{orthogonality}, the high-order quasi-Helmholtz projectors satisfy 
\begin{equation}\label{cancellation}
    \mat P_p^{\Lambda H} \mat T_h = \mat 0 \quad \text{and} \quad \mat T_h \mat P_p^{\Lambda H} = \mat 0,
\end{equation}
which generalizes the analogous property valid in the zeroth-order case \cite{Adrian2021}. This allows to proceed with the same formal strategy developed for zeroth-order projectors that here will work for the high-order EFIE. Define
 the preconditioning matrix
\begin{equation}
    \mat P  = j \sqrt{k / C} \, \mat P_p^\Sigma + \sqrt{C/k} \, \mat  P_p^{\Lambda H},
\end{equation}
with the scaling factor
\begin{equation}\label{scaling-constant}
    C = \sqrt{ \dfrac{ \|\mat T_h\|}{\| \mat  P_p^{\Lambda H} \mat T_s  \mat P_p^{\Lambda H} \|} }.
\end{equation}
The preconditioned system reads
\begin{equation}\label{eq:stableeq}
    \mat P \mat T  \mat P \vect y =  \mat P \vect e,
\end{equation}
with 
\begin{equation} 
    \vect j = \mat P \vect y.
\end{equation}
The proof of low-frequency well-conditioning and stability of \eqref{eq:stableeq} is formally identical to the one for the zeroth-order case \cite{Adrian2021} and we omit it here for the sake of conciseness.

\section{Implementation Related Details}\label{implementation}

\color{black}In this section, we deal with  implementation related details that could be useful to the reader while implementing all new techniques described here.

We assume that all matrix-vector products are done in a quasi-linear number of operations, using a compression approach~\cite{pan_efficient_2012, Rjasanow_2013} for the integral operators and a sparse algorithm for the Star matrices. Note that, to match the computational complexity of fast matrix-vector product algorithms, the naive computation of the norms in \eqref{scaling-constant} should be avoided and iterative approaches, such as power iterations \cite{power_iteration}, should be used instead.
Moreover, the numerical (pseudo-)inversion of $\mat \Sigma_p^T \mat \Sigma_p$ in \eqref{p-nsol-proj} has to be done iteratively whenever a product with the projectors is involved. As in the zeroth-order case, this matrix is similar to the one we get by discretizing the Laplacian equation with finite elements of order $p$. 
We can thus rely on algorithms tailored to solve such systems from other fields of applications~\cite{Helenbrook}. In particular, multigrid algorithms have proven to be quite effective. Because there exists several variants of multigrids, one could use the $p$-multigrid versions~\cite{Tielen2020, Helenbrook, Huismann2019} (tailored for high-order finite elements) but there are also robust implementations of the standard algebraic multigrid~\cite{Napov2014} generalized to different scenarios that are often easier to use in a black-box fashion.

Finally, to maximize the effect of our regularization to all frequencies, the computation of the projected EFIE has to be done carefully. For numerical purposes, the cancellations \eqref{cancellation} are to be enforced explicitly in the computation of the preconditioned matrix~\cite{Adrian2021}, which reads
\begin{equation}
\begin{aligned}
    \mat P \mat T  \mat P &= j C \, \mat P_p^{\Lambda H} \mat T_s \mat P_p^{\Lambda H} + j/C \, \mat T_h
    - k \, \mat P_p^{\Lambda H} \mat T_s \mat P_p^\Sigma \\
    &- k  \, \mat P_p^\Sigma \mat T_s \mat P_p^{\Lambda H} - j k^2/C \, \mat P_p^\Sigma \mat T_s \mat P_p^\Sigma.
\end{aligned}
\end{equation}
The loss of significant digits in the computation of the right-hand side also has to be avoided, as is well-known when dealing with very low-frequency numerical strategies \cite{Adrian2021}. 
In particular, the solenoidal and non-solenoidal contributions of the right-hand side must be computed separately, and the static part of the excitation source is subtracted when tested with the solenoidal functions~\cite{Chew2000, Adrian2021}.
More explicitly, we write
\begin{equation}
    \mat P \vect e = j \sqrt{k/C} \, \mat P_p^\Sigma \vect e + \sqrt{C/k} \, \mat P_p^{\Lambda H} \vect e_\mathrm{sub},
\end{equation}
where $\vect e_\mathrm{sub}$ is the subtracted righ-hand side \cite{Adrian2021}.
Finally, the post-processing computation of the electric field requires a similar treatment, by using separately the solenoidal and non-solenoidal contributions of the solution 
\begin{equation}
    \vect j_\mathrm{sol} =  \sqrt{C / k} \, \mat P_p^{\Lambda H} \vect y \quad \text{and} \quad
    \vect j_\mathrm{nsol} =  j \sqrt{k / C} \, \mat P_p^{\Sigma} \vect y
\end{equation}
and subtracting the static part of the Green's kernel when integrating with the solenoidal contribution~\cite{Adrian2021}.\color{black}

\section{Numerical validation}\label{numerical-results}

In this section, we give numerical results that validate the use of the quasi-Helmholtz projectors by comparing it against the standard EFIE and against a standard quasi-Helmholtz decomposition.
The  quasi-Helmholtz decomposition we adopted for comparison is a generalization of the Loop/Star technique (L/S-EFIE) from the zeroth-order: the matrix $\mat \Sigma_p$ is used for generating the Stars and the Loops are computed as described in~\cite{wildman_accurate_2004}.
The meshes are generated from the software Gmsh that provides quadratic (curvilinear) triangles \cite{gmsh}.
The singular integrals are computed using the singularity cancellation scheme described in \cite{sauter2010boundary} while near-singular and far interactions are computed with Gaussian quadratures.

Our first validation is done over the unitary sphere for different frequencies. 
Figure~\ref{condNum-sphere} shows the condition number of system matrix for each different method as a function of frequency and for different orders.
We observe that the condition number of the standard EFIE increases dramatically, as expected, while it remains constant for the standard Loop/Star technique and for the approach proposed here based on high-order quasi-Helmholtz projectors. For the latter case, however, the condition number is lower.
This effect can be better understood from the test in Figure~\ref{condinfuncofh} where  the condition number is shown in function of the inverse mesh size. It is clear that, differently from standard Loop/Star, the new projectors do not worsen the spectral behavior of the original EFIE generalizing what happens in the zeroth-order case~\cite{Adrian2021}.
The  radar cross sections (RCS)  at $100$Hz  and at $3 \times 10^8$Hz are obtained in  Figures~\ref{RCS-sphere_100Hz} and~\ref{RCS-sphere_3e8Hz}, respectively. This shows that the low frequency-breakdown, absent at higher frequencies, is instead corrupting the results at low-frequency when no low-frequency treatment is employed. At the same time, these results validate the stability of our scheme in a wide frequency range.
\begin{figure*}
\begin{minipage}[ht]{\linewidth}
      \centering
      \begin{minipage}{0.49\linewidth}
          \begin{figure}[H]
          \centering
              \resizebox{0.95\columnwidth}{!}{%
    \includegraphics{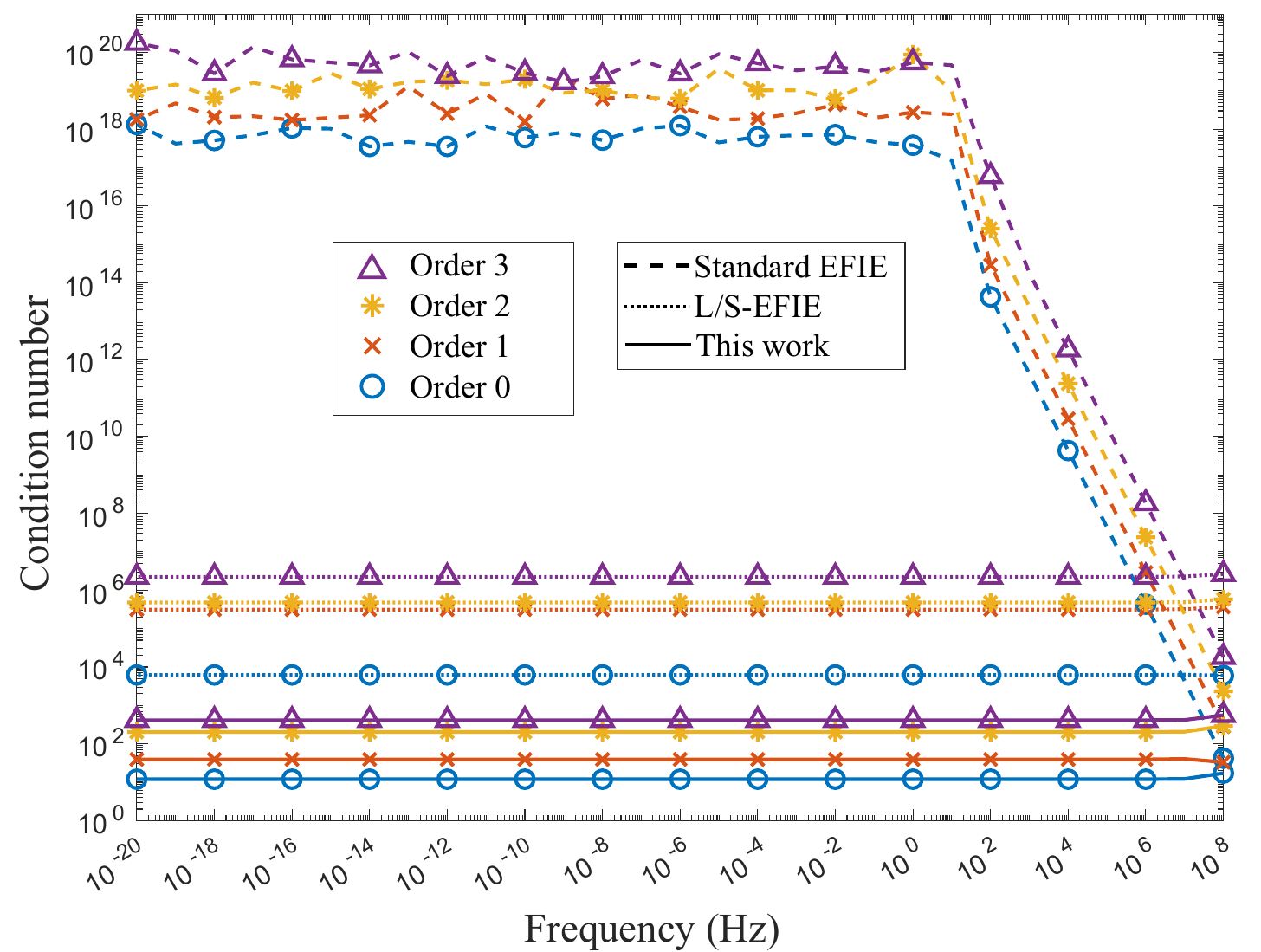}
    }\caption{Condition number of the system matrices in function of the frequency for order zero to three.}\label{condNum-sphere}
          \end{figure}
      \end{minipage}
      \hfill
      \begin{minipage}[h]{0.49\linewidth}
          \begin{figure}[H]
          \centering
              \resizebox{0.95\columnwidth}{!}{%
    \includegraphics{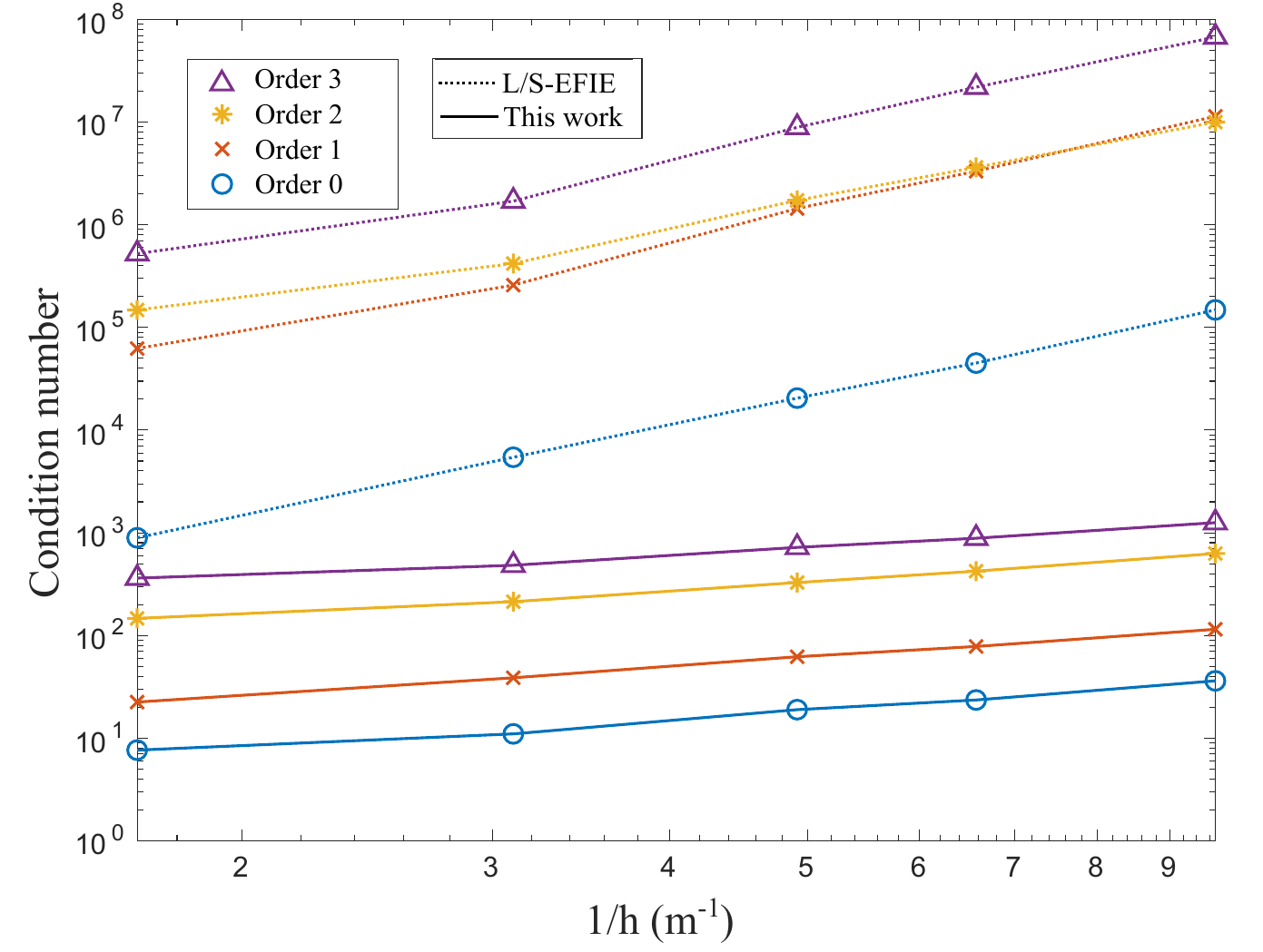}
    }\caption{Condition number of the system matrices in function of the inverse of the average cell diameter $h$; frequency of 1Hz.}\label{condinfuncofh}
          \end{figure}
      \end{minipage}
  \end{minipage}
\end{figure*}

\begin{figure*}
\begin{minipage}[ht]{\linewidth}
      \centering
      \begin{minipage}{0.49\linewidth}
          \begin{figure}[H]
          \centering
             \resizebox{0.85\columnwidth}{!}{
    \includegraphics{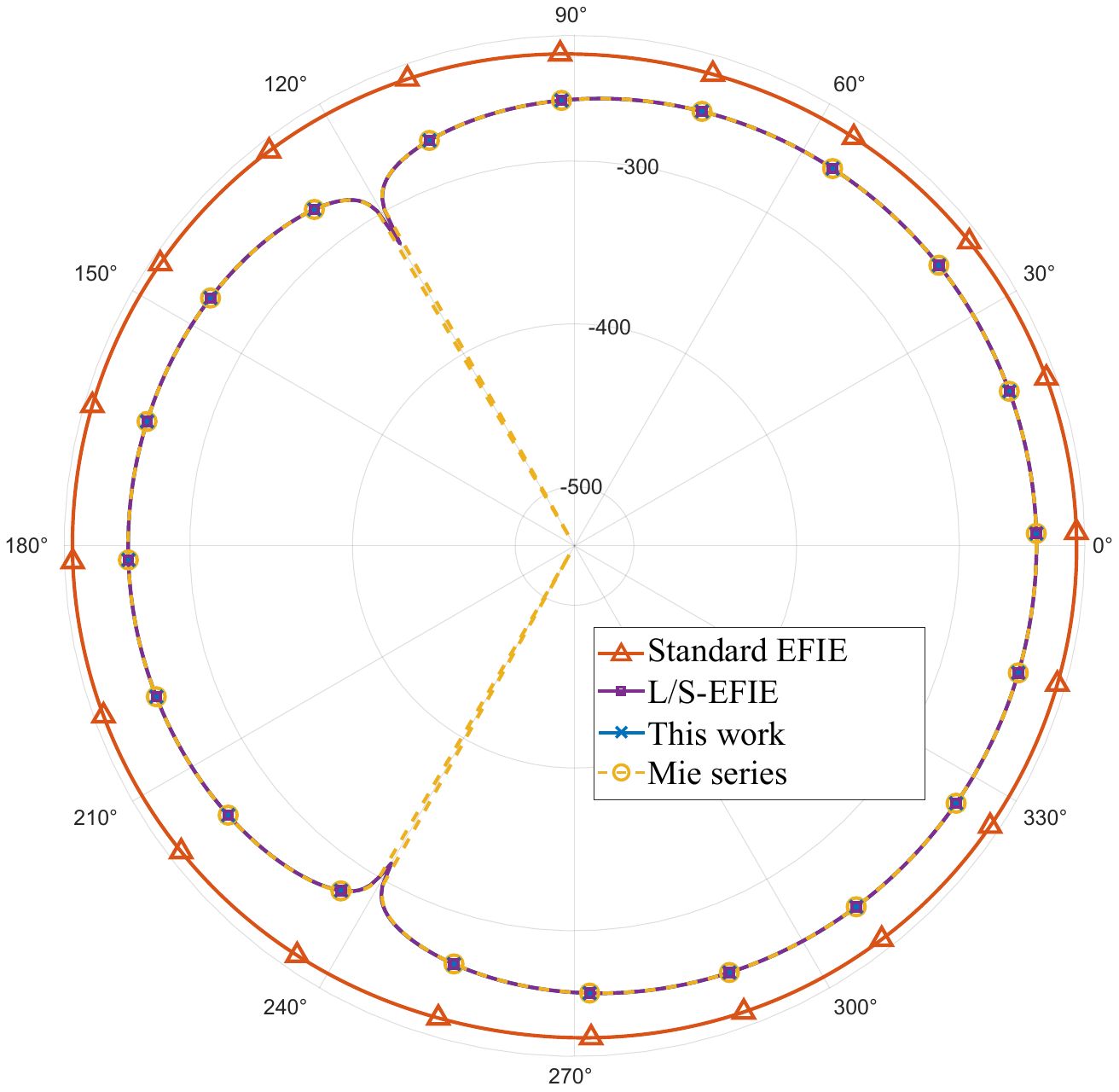}
    }\caption{Radar cross section (RCS) in dBsm at 100Hz.}\label{RCS-sphere_100Hz}
          \end{figure}
      \end{minipage}
      \hfill
      \begin{minipage}[h]{0.49\linewidth}
          \begin{figure}[H]
          \centering
              \resizebox{0.85\columnwidth}{!}{%
    \includegraphics{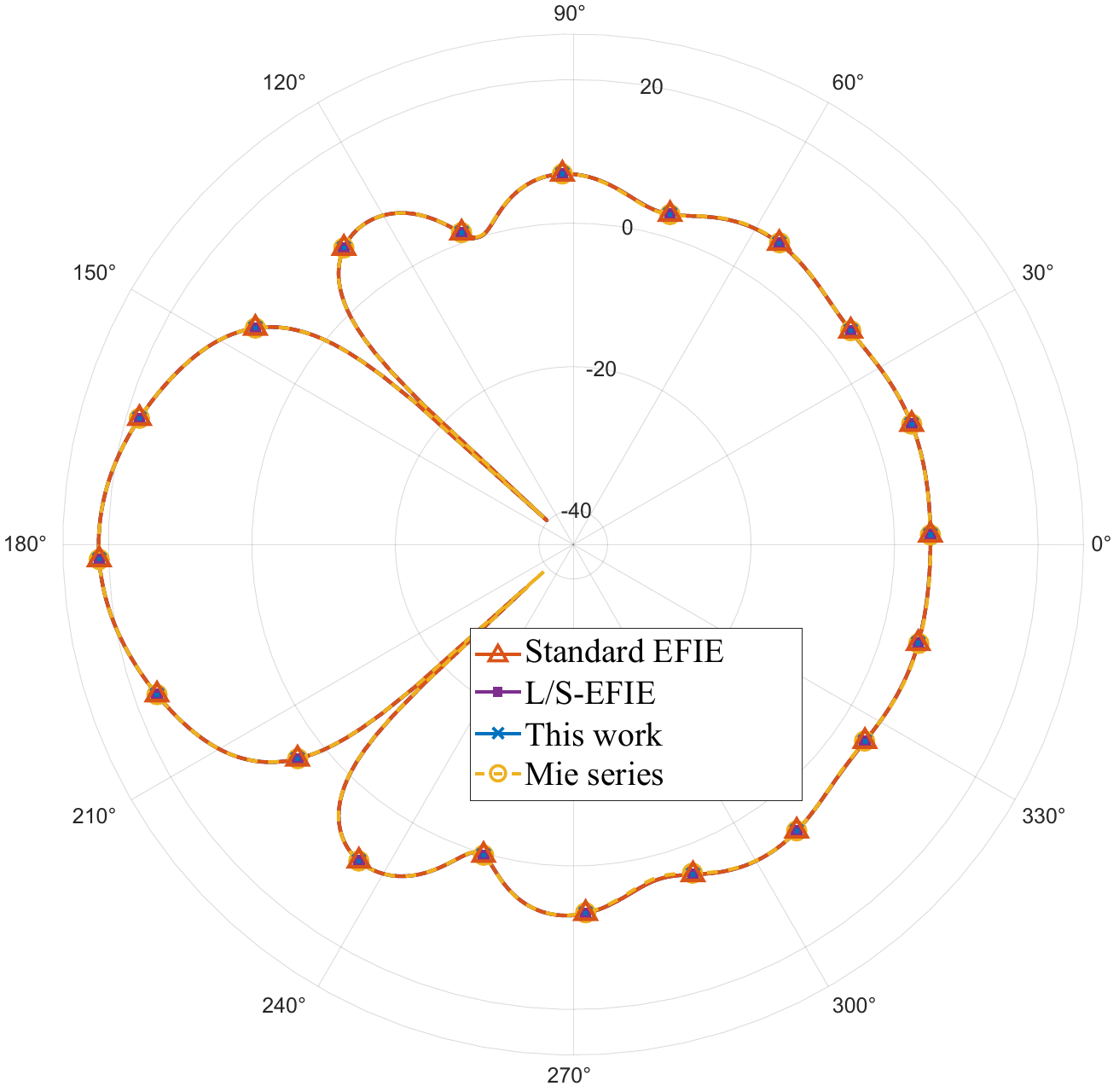}
    }\caption{Radar cross section (RCS) in dBsm at $3 \times 10^8$Hz.}\label{RCS-sphere_3e8Hz}
          \end{figure}
      \end{minipage}
  \end{minipage}
\end{figure*}

To test the performance of our scheme on a non-simply connected geometry we used the ``two interlocked Möbius ring'' structure, that form two separate objects with interlaced holes and handles.
The rings are of radius $1$m and are discretized with 3780 cells and 5670 edges.
In total, there are 72 global  Loops associated with this topology.
Figure~\ref{Mobius} shows the surface density current obtained by solving the EFIE with and without preconditioning at $10$Hz using basis functions of order two (39690 unknowns). 
We get, without the need of detecting the global Loops, the same results with the quasi-Helmholtz projectors and the Loop/Star-EFIE for which however the global Loops must be explicitly detected. It should also be noted that the accuracy is completely lost without preconditioning. 

\begin{figure*}[ht]
    \centering
    \resizebox{0.88\textwidth}{!}{\includegraphics{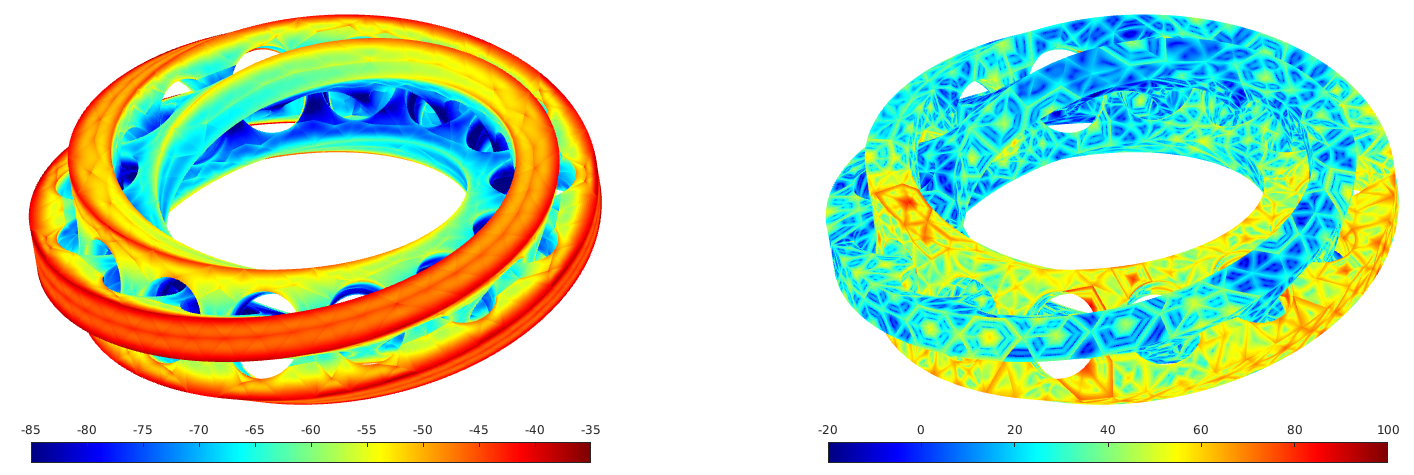}
    }
    \caption{Absolute value of the surface density current (in $\text{dB}_{\text{A/m}^2}$) for the interlocked Möbius ring irradiated by a plane wave at $10$Hz, with preconditioning (up) and without (down).}
    \label{Mobius}
\end{figure*}

Our final validation test scenario is on the model of an airliner with 2D apertures corresponding to the windows and the shell of the jet engine. The mesh contains 3942 cells and 5750 internal edges, and we use basis functions of order two (42900 unknowns).
In addition to show the relevance of our method to solve industrial scenarios, this example completes the numerical study with problems containing global Loops around apertures, which are of different kind than the harmonic Loops of the previous example. 
Figure~\ref{airliner} shows the surface density current obtained by solving the EFIE with and without preconditioning at $1$Hz. 
As for the previous case, the quasi-Helmholtz projectors give an identical solution as the one obtained with Loop-Star decompositions (which however requires Loop detection), while the solution without preconditioning is completely jeopardized.

\begin{figure*}
    \centering
    \resizebox{0.89\textwidth}{!}{\includegraphics{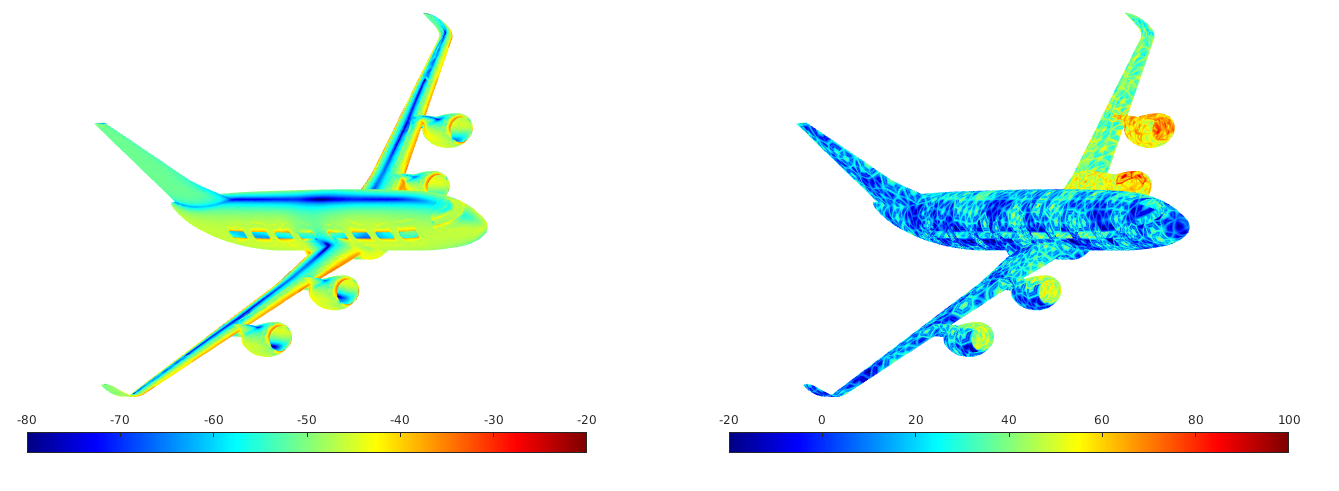}
    }
    \caption{Absolute value of the surface density current (in $\text{dB}_{\text{A/m}^2}$) for the airliner irradiated by a plane wave at $1$Hz, with preconditioning (left) and without (right).}
    \label{airliner}
\end{figure*}

\section{Conclusion}

In this work, we have extended the use of the quasi-Helmholtz projectors for stabilizing the EFIE when discretized with high-order basis functions. The new scheme is based on a generalized definition of Star matrix that is wide enough to encompass numerous relevant scenarios, including the use of GWP basis elements. The contribution has shown that this results into unique quasi-Helmholtz projectors irrespectively of the specific choice of Star matrix.
Numerical results has shown the effectiveness of the new approach.

\section*{Acknowledgment}

The work of this paper has received funding from the EU
H2020 research and innovation programme under the Marie
Skłodowska-Curie grant agreement n° 955476 (project COM-
PETE), from the European Research Council (ERC) under
the European Union’s Horizon 2020 research and innovation
programme (grant agreement n° 724846, project 321).

\ifCLASSOPTIONcaptionsoff
  \newpage
\fi

\bibliographystyle{IEEEtran}
\bibliography{references}

\end{document}